\documentclass[12pt]{article}
\usepackage{geometry}                
\usepackage{graphicx}
\usepackage{amssymb}
\usepackage{epstopdf}
\usepackage{graphicx}
\usepackage{amssymb}
\usepackage{epstopdf}
\usepackage{amsmath}
\usepackage{amsthm}
\usepackage{amsfonts}
\usepackage{epstopdf}

\newtheorem{theorem}{Theorem}
\newtheorem{lemma}[theorem]{Lemma}

\newtheorem{claim}[theorem]{Claim}

\newtheorem{proposition}[theorem]{Proposition}
\theoremstyle{definition}

\theoremstyle{remark}

\numberwithin{equation}{section}

\newcommand\RR{{\mathbb R}}

\def\W2{W^{1,2}({\cal O}(M))}

\def\Ai{\mbox{Ai}}

\def\b{\beta}
\newcommand{\m}{\mathfrak{m}}

\newcommand{\ra}{\rightarrow}

\newcommand\one{{\mathbf 1}}

\newcommand{\stb}{\mbox{\small $\frac{2}{\sqrt{\beta}}$}}

\newcommand{\tb}{\mbox{\small $\frac{2}{\beta}$}}

\title{Diffusion at  the random matrix hard edge}
\author{Jos\'e A. Ram\'{\i}rez 
\and Brian Rider}
\date{March 13, 2008}                                           

\begin{document}
\maketitle

\begin{abstract} 
We show that the limiting minimal
eigenvalue distributions  for a natural generalization of Gaussian sample-covariance structures
(``beta ensembles") are described  by the spectrum of a random diffusion generator.
This generator may be mapped onto the ``Stochastic Bessel Operator," introduced and studied 
by A. Edelman and B. Sutton in \cite{ES} where the corresponding convergence was first conjectured.
Here, by a Riccati transformation,  we also obtain a second diffusion description of the limiting eigenvalues
in terms of hitting laws.  All this pertains to the so-called hard edge of random matrix theory
and sits in complement to the recent work \cite{RRV} of the authors and B. Vir\'ag on the general 
beta random matrix soft edge.  In fact, the diffusion descriptions found on both sides are used below 
to prove there exists a transition between the soft and hard edge laws at all values of beta.

\end{abstract}

\section{Introduction}

The origins of random matrix theory can be traced to the introduction of Wishart's ensembles, matrices
of the form $XX^{\dagger}$ with rectangular $X$ comprised entirely of independent real or complex
Gaussians of mean zero and mean-square one.  The spectrum of these objects are of fundamental 
importance in mathematical statistics (see the comprehensive text \cite{Muir}), and continue to generate
wide interest due to their relevance to such disparate areas as information theory \cite{Tel}, numerical analysis \cite{Edelman},
and, along with their quaternion-entried counterparts, theoretical physics \cite{Verb}.

Here we consider scaling limits for Wishart-type eigenvalues at the hard edge.  To explain, let $X$ be $n \times m$. If $m \simeq n$
as $n \uparrow \infty$ the minimal eigenvalues of the (non-negative) $X X^{\dagger}$ will feel the ``hard" constraint at the origin, 
while if $m/n$ is strictly larger than one in the large dimensional limit, the minimal eigenvalues separate form zero and one has 
``soft" edge fluctuations (on which more below). In fact, if $m= n+a$ with fixed $a$ as $n \uparrow \infty$ one discovers an interesting
 family of limit laws indexed by $a$ for the bottom of the spectrum.

The known results at the hard edge have thus far been based on the explicit  joint density for the Wishart eigenvalues
$0 \le \lambda_0, \lambda_1, \cdots , \lambda_{n-1}$.  In particular, when $X$ is $n \times (n+a)$ with integer $ a > -1$, that density is
\begin{equation}
\label{LagDensity}
  P_{\beta, a} (\lambda_1, \dots, \lambda_n) = \frac{1}{Z_{\beta, a}}  \prod_{j < k} |\lambda_j - \lambda_k|^{\beta} \times 
     \prod_{k=0}^{n-1}    \lambda_k^{\frac{\beta}{2}(a+1) - 1}  e^{-\frac{\beta}{2} \lambda_k},
\end{equation}
with normalizer $Z_{\beta, a} < \infty$ and $\beta = 1, 2,$ or $4$ for real, complex, or quaternion Gaussian entries.
More importantly, with these choices of $\beta$ all finite dimensional correlation functions of the eigenvalues are
computable in terms of Laguerre polynomials (thus the common tag ``Laguerre ensembles").  At $\beta = 2$ and all
valid $a$, \cite{TW2} proves the limiting distribution of the minimal eigenvalue is described by the Fredholm determinant of
a kernel operator given in terms of Bessel functions, and based on this derives a second description of the limit law
as  a functional of the fifth Painlev\'e transcendent.  Other work at $\beta = 1, 2, 4$ hard edge include \cite{DeiftVan}, \cite{For2},
and \cite{Vanlessen}.  These again rely on the underlying orthogonal polynomial structure (the first and third reference use 
Riemann-Hilbert methods to replace the exponential weight $e^{-(\beta/2) \lambda}$ in (\ref{LagDensity}) with a more general
$e^{-V(\lambda)}$ potential), and describe the eventual limit law through Fredholm determinants or Fredholm pfaffians.

While the distribution on $n$ points $\lambda_1, \dots, \lambda_n \in \RR_+$  
defined by
(\ref{LagDensity}) makes sense for all $\beta > 0$ and $a > -1$, the orthogonal polynomial approach breaks 
down outside the standard triple in $\beta$.  For some special choices of the parameters beyond $\beta = 1,2,4$, \cite{For1} was
able to exploit the niceties of the exponential weight to obtain limit laws in terms of hypergeometric funcions.   Still,
even the existence of the general $\beta$ hard edge limit law remained open until now.

Our approach, similar to that in \cite{RRV}, rests on the existence of tridiagonal matrix models for 
all $\beta$.  Set for any $a > -1$ and $\beta > 0$,
\begin{equation}
\label{TriModel}
L_{\beta,a}= \frac{1}{\sqrt{\beta}} \left[ \begin{array}{ccccc}
\chi_{(a+n)\beta} & \chi_{(n-1)\beta} & & & \\
& \chi_{(a+n-1)\beta} & \chi_{(n-2)\beta} & & \\
& & \ddots & \ddots & \\
& & & \chi_{(a+2)\beta} & \chi_{\beta} \\
& & & &\chi_{(a+1)\beta}  \\
\end{array}
\right]
\end{equation}
in which each $\chi_r$  that appears is an independent $\chi$ random variable of the
indicated  index.  (We suppress here the dimension parameter $n$ on the $n \times n$
random matrix $L_{\beta, a}$).  Then, as discovered by Dumitriu and Edelman \cite{DE}, the eigenvalues
of $L_{\beta, a} L_{\beta, a}^T$ have (\ref{LagDensity}) as their joint density  function.   Note that, when $\beta = 1$
or $2$, the bidiagonal (\ref{TriModel}) may be arrived at by performing Householder transformations on the corresponding
``full" Wishart ensemble; this fact was used previously in a random matrix context by Silverstein \cite{Silv}.

Viewing the $n \uparrow \infty$ limit as giving rise to a continuum approximation to the discrete operators $L_{\beta, a}$, 
an entry-wise expansion in the random $\chi$ variables led Edelman and Sutton
to the following  conjecture for the full $\beta > 0$ hard edge.

\medskip

\noindent
{\bf{Conjecture}}  ({\em{Edelman-Sutton}} \cite{ES})  {\em{Let}} $s_k$  {\em{denote the}} $k$-th  {\em smallest singular value of  the bidiagonal 
operator} $L_{\beta, a}$.  {\em{Then, as}} $n \uparrow \infty$ {\em{the family}} $\{ \sqrt{n} s_k \}$ 
{\em converges in law to the corresponding singular values of }
$$
   \mathcal{L}_{\beta, a}  = -  \sqrt{x} \frac{d}{dx} + \frac{a}{2 \sqrt{x}}   +  \frac{1}{\sqrt{\beta}} b^{\prime}(x)
$$
{\em{in which}} $x \mapsto b(x)$ {\em{is a Brownian motion}}.  {\em{Here,}} $\mathcal{L}_{\beta, a}$
{\em{is understood to act on functions}} 
$f \in L^2[0,1]$ {\em{subject to}} $f(1) = 0$ {\em{and}} $(\mathcal{L}_{\beta, a} f)(0) = 0$.

\medskip

The random $ \mathcal{L}_{\beta, a}$ was tagged the Stochastic Bessel Operator in \cite{ES}
on account of the zero-noise ($\beta = \infty$) version having singular-values at the roots of  $J_a$
(the Bessel function of the first kind).

Our main result establishes this conjecture, though we prefer to phrase matters in a different way,
back in terms of eigenvalues of the symmetric ensembles $L_{\beta, a}^{} L_{\beta, a}^T$, henceforth referred to as the 
$(\beta,a)$-Laguerre ensembles.  Toward this,
introduce the random operator of second order,
\begin{equation}
\label{Operator1}
 \mathfrak{G}_{\beta, a} = -  \exp{ [  (a+1)x + \stb b(x) ]} \frac{d}{dx} \Bigl\{ \exp{[ - a x -  \stb b(x) ]} \frac{d}{dx}  \Bigr\},
\end{equation}
where again  $b(x)$ is a Brownian motion and $a > -1$, $\beta >0$. 
Formal manipulations will take you from $\mathcal{L}_{\beta, a}^{} \mathcal{L}_{\beta, a}^T$
to $\mathfrak{G}_{\beta, a}$, but the
latter is better understood upon recognizing, in the spirit of the title,  
that $ -\mathfrak{G}_{\beta, a}$ generates the diffusion
with (random) speed and scale measures
$$
  m(dx)  = e^{-(a+1) x - \stb b(x) } \, dx  \   \mbox{ and }   \  s(dx) =  e^{ax + \stb b(x)  } \, dx.
$$
This motion may be built pathwise in the classical mode (see for example \cite{IM}),
placing (\ref{Operator1}) on firm ground.  
The limiting spectral problem will require consideration 
of $\mathfrak{G}_{\beta, a}$  acting on the positive half-line  
with Dirichlet conditions at the origin, and this carries over  into 
killing the underlying process when reaching that point. 

Even more convenient, we may
define eigenvalues/eigenvectors through the resolvent equation.   That is, if we at first 
take the equation $\mathfrak{G}_{\beta, a} \psi = \lambda \psi$ to mean $ \psi = \lambda
\mathfrak{G}_{\beta, a}^{-1} \psi $,   the speed and scale construction provides
the explicit form of the inverse,
\begin{equation}
\label{Greensfunction}
  (\mathfrak{G}_{\beta, a}^{-1} \psi )(x) \equiv \int_0^{\infty}  
     \left( \int_0^{x \wedge y}  s(dz)  \right) \,  \psi(y)  \,  m(dy). 
\end{equation}
Now $ \mathfrak{G}_{\beta, a}^{-1} $ is plainly non-negative symmetric in $L^2[ \RR_+, m]$ 
and the Dirichlet condition at the origin is automatic for solutions of $ \psi = \lambda
\mathfrak{G}_{\beta, a}^{-1} \psi $.  Lying slightly deeper, we will see that (almost surely)  
$\mathfrak{G}_{\beta, a}^{-1}$ maps $L^2[\RR_+, m]$  into $C^{3/2-}$ and is in fact of trace class. 
We  have:

\begin{theorem}
\label{mainthm}   With probability one,  when restricted to the positive
half-line with Dirichlet conditions at the origin, $\mathfrak{G}_{\beta,a}$ has  discrete spectrum
comprised of simple eigenvalues $0 < \Lambda_0(\beta, a) < \Lambda_1(\beta,a)
< \cdots \uparrow \infty$.   Moreover, 
with now
 $0< \lambda_0 < \lambda_1 < \cdots  < \lambda_n$ the ordered $(\beta,a)$-Laguerre
eigenvalues, 
$$
   \{  n \lambda_0, n \lambda_1, \dots , n \lambda_k  \}   \Rightarrow \{ \Lambda_0(\beta, a), 
      \Lambda_1(\beta, a) , \dots, \Lambda_k(\beta, a) \}
$$
(jointly in law)
for any fixed $k < \infty$ as $n \uparrow \infty$.
\end{theorem}

\medskip

\noindent
{\em Remark. }  The Dirichlet condition for $\mathfrak{G}_{\beta, a}$ at $x=0$ may be mapped
to that at $x=1$ for $\mathcal{L}_{\beta, a}$ in the Edelman-Sutton conjecture.  On the other hand,
the process generated by $- \mathfrak{G}_{\beta, a}$ has a natural (or free) boundary at $x=+\infty$,
which carries certain advantages over  the specified condition for $\mathcal{L}_{\beta, a}$ at
$x=1$ in the conjecture.

\medskip

As a bit of amplification, differentiating with abandon one is led to 
$$
  - \mathfrak{G}_{\beta,a} = e^{x}   \Bigl(  \frac{d^2}{dx^2} - ( a + \stb b^{\prime}(x) ) \frac{d}{dx} \Bigr),
$$
along with the idea that the corresponding motion 
is  just a Brownian motion with (shifted) white noise drift.   In fact, modulo
the multiplicative factor $e^x$ which affects a change of time, this is precisely the random diffusion
introduced by Brox as a continuum analogue of Sinai's walk \cite{Brox}.   Theorem \ref{mainthm} 
then draws a 
concrete connection between random matrix theory and the  lifetime of this random process in a random
environment which has been the subject of continued investigation since its introduction (see \cite{Talet} and the many references
within, or the recent \cite{Bovier} for a spectral point of view).

Our second description of the hard-edge is a corollary of the first, employing  Riccati's  map 
to transform a solution of, the suitably interpreted, $ \mathfrak{G}_{\beta,a}  \psi(x,\lambda) =
 \lambda \psi(x,\lambda) $ for any fixed $\lambda \ge 0$
into one of
\begin{equation}
\label{diff}
   dp(x)  = \stb p(x) db(x) + \left( (a + \tb ) p(x) - p^2(x) - \lambda e^{-x} \right) dx,
\end{equation}
understood in the sense of It\^{o}. The point is: Sturm's oscillation theorem implies 
that the eigenvalues  of $ \mathfrak{G}_{\beta,a}$ are counted by the zeros of $\psi(\cdot, \lambda)$,
and those zeros correspond to places 
where $p \equiv \psi^{\prime}/\psi$, which solves  (\ref{diff}), hits $-\infty$.

\begin{theorem}
\label{DiffThm}
Let $P_{x, c} $ denote the law of $p(\cdot) = p(\cdot; a, \beta,\lambda)$ 
 starting from position $c$
at time $x$.  Let also $\nu_{x}(dc) = P_{x, +\infty} ( \mathfrak{m} \in dc )$ where
$\mathfrak{m}$ is the passage time of $p$ to $-\infty$. Then,
\begin{eqnarray*}
  P( \Lambda_0(\beta,a) > \lambda) & = &  \nu_{0}(\{\infty\})   \mbox{ and, more generally, } \\
  P (\Lambda_k(\beta, a) < \lambda) & =  & \int_{\RR^{k+1}} \nu_0(dx_1) \nu_{x_1}(dx_2) \cdots
     \nu_{x_{k}}(d x_{k+1}).
\end{eqnarray*}
\end{theorem}

Theorem 1 is proved in Section 2, and Theorem 2 in Section 3.   We conclude the introduction by describing a general transition 
between the hard edge laws just described and the form of the $\beta > 0$ soft edge laws established in \cite{RRV}. 

\medskip

\noindent
{\em Remark.}  It is natural to ask whether from Theorems 1 and 2 one might recover the Painlev\'e or Hypergeometric 
descriptions of the hard edge from \cite{TW2} or \cite{For1} respectively, and then go further by finding explicit formulas
of the distributions at all $\beta >0$.  Thus far the answer is no,   even when $(a+1) = 2/\beta$
and  $\Lambda_0(a, \beta)$ is just an exponential random variable (as is easily seen  from the joint density  (\ref{LagDensity})).
With that choice of parameters,  the speed  $m'(x) = e^{ - (a+1) x -\stb b(x)}$ of the $\mathfrak{G}_{\beta, a}$-diffusion turns out to be a martingale, 
but we do not see how to make use of this.

\subsection*{Soft edge and transition}

The random matrix soft edge corresponds to the scaling limits of the maximal, rather than minimal, eigenvalues in the Laguerre
ensembles.  Historically, these laws were discovered first by Tracy and Widom (\cite{TW1} and \cite{TW3}) in the context of a different 
class of random matrices, the Gaussian  Orthogonal, Unitary, and Symplectic ensembles. The latter are $n \times n$ real symmetric, complex
hermitian, or quaternion self-dual matrices with Gaussian entries.  There is again an explicit joint spectral density, which takes the form
of a constant multiple of $\prod | \lambda_i - \lambda_j|^{\beta}  e^{- (\beta/4) \sum  \lambda_i^2}$ with $\beta = 1, 2$ or $4$ respectively.
And once more, all correlations are given in terms of orthogonal polynomials (now Hermites).  Tracy and Widom 
proved that the appropriately scaled largest eigenvalues have distribution functions  described by Painlev\'e II (via a more basic
formulation in terms of Fredholm determinants/pfaffians of an Airy kernel). Matters were later carried over to the Laguerre soft edge 
by a collection of authors.

As before, one may consider the general ``$\beta$-Hermite" laws.   \cite{DE} provides a separate family of tridiagonal matrix models
for these laws (though see also \cite{Trotter} for an earlier application at $\beta=1,2$ ), and in direct analogy with Theorems 1 and 2 the authors and B. Vir\'ag have previously proved:

\medskip
\noindent
{\bf Theorem} ({\em{Theorems 1.1 and 1.2 of}} \cite{RRV}) {\em The largest eigenvalues in either the} $\beta$-{\em{Laguerre or}}
$\beta$-{\em{Hermite ensembles have scaling limits given by the law of the top eigenvalues
of the random Schr\"odinger operator}} 
$
   - \mathcal{H}_{\beta} =  \frac{d^2}{dx^2} - x + \stb b'(x).
$ 
{\em There is also a description of the limiting soft-edge eigenvalues equivalent to 
that in Theorem 2, with the diffusion}
\begin{equation}
\label{olddiff}
  dp(x) = \stb db(x) + (\lambda + x - p^2(x)) dx
\end{equation}
{\em in place of}  (\ref{diff}).
\medskip

The first part of this result, the identification of $-\mathcal{H}_{\beta}$ at the soft edge, proves a different conjecture of Edleman and Sutton from \cite{ES}.
For obvious reasons, we refer to the distribution of the top eigenvalue of $-\mathcal{H}_{\beta}$ as the general beta Tracy-Widom law,
notated $TW_{\beta}$.

Returning to the $(\beta, a)$-Laguerre ensembles, it is well understood that if $a$ tends to infinity with $n$ so that $\lim_{n \rightarrow \infty} m/n = \lim_{n \rightarrow \infty} (n+a)/n > 1$,
the limiting spectral measure is pulled away from the origin and one sees soft-edge behavior at both the minimal and maximal eigenvalues.
Thus, one expects that by taking $a \rightarrow \infty$ after $n \rightarrow \infty $, the hard-edge becomes a soft-edge and creates a link between
these families of distributions arising in random matrix theory.  Borodin and Forrester \cite{BF} have shown that this is indeed the case in the classical
$\beta = 1, 2$ and $4$ settings.  Their work rests on the aforementioned  determinantal 
 forms of the underlying distribution functions.
Employing just the diffusions (\ref{diff}) and (\ref{olddiff}), our final result shows that this transition holds true at all $\beta > 0$.

\begin{theorem}  
\label{h2sthm}
With $\Lambda_0(\beta,a)$ the limiting smallest eigenvalue in the
$(\beta,a)$-ensemble and $TW_\beta$ the general beta Tracy-Widom law,
$$
    \frac{  \eta - \Lambda_0(\beta, 2 \sqrt{\eta} - \tb )} {  \eta^{2/3}} \Rightarrow TW_{\beta}
$$
as $\eta \ra \infty$.
\end{theorem}

Theorem \ref{h2sthm} is proved in Section 4.

In summary, together with \cite{RRV} the present provides a complete picture of the extremal laws of random matrix theory, at all
values of the natural parameters. This leaves apart the general $\beta$ spectral bulk, which has recently been 
treated by Valko-Vir\'ag  \cite{ValkoVirag} (for the $\beta$-Hermite ensembles) and
Killip-Stoiciu \cite{Kill} (for the circular $\beta$ ensembles, 
generalizing the eigenvalue laws for the Haar distributed unitary group).

\section{Convergence of the spectrum}

The key is to prove the almost sure strong convergence of the resolvent operators, 
or really a similarity transformation of the sequence of $(L_{\beta, a}^{} L_{\beta, a}^T)^{-1}$ 
matrices to a version of ${\mathfrak{G}}_{\beta, a}^{-1}$.   As we will see, all these objects 
may be viewed as integral operators with well-behaved kernels, allowing for an efficient
verification of the necessary compactness.

\subsection*{Outline}

Set $M_{\beta, a} = S L_{\beta, a} S^{-1}$ where $S$ is the anti-diagonal matrix
of alternating signs $S_{ij} = (-1)^i \delta_{i+j - n-1}$.  The spectrum is unchanged
(we may work with $M_{\beta, a} M_{\beta, a}^T$ rather than $L_{\beta, a} L_{\beta, a}^T$),
and we record
$$
   M_{\beta, a } = \frac{1} {\sqrt{\beta}}  
 \left[ \begin{array}{ccccc}
\chi_{(a+1)\beta} &   & & & \\
- {\tilde{\chi}}_{\beta} & \chi_{(a+2)\beta} &   & & \\
&    -  {\tilde{\chi}}_{2 \beta} &  \chi_{(a+3) \beta}  &  & \\
& & \ddots &  \ddots &   \\
& & &   - \tilde{\chi}_{(n-1) \beta} &  \chi_{(a+n) \beta}   \\
\end{array}
\right],
$$
where the additional notation is intended to emphasize the independence of the processes
along the main and lower diagonals.

Wishing to track inverses, we first note the readily checked fact:

\begin{lemma}  For any lower bidiagonal matrix $B = b_{i,j}$ (that is, $b_{i j} = 0$ if $j> i$ or $j < i-1$),
the inverse, when it exists, is lower triangular and has the expression
$$
   [B^{-1}]_{i,j}  =  \frac{(-1)^{i+j}}{b_{ii}} \prod_{k=j}^{i-1} \frac{b_{k+1, k}}{b_{k,k}} \  \  \   \mbox{ for } j \le i.
$$ 
\end{lemma}

Next, observe that for any $A = a_{i,j} \in \RR^{n \times n}$ there is
a natural operator embedding 
into $L^2[0,1]$ which does change the spectrum:
$$
     (A f) (x)  \equiv  \sum_{j=1}^n   a_{i,j} n \int_{x_{j-1}}^{x_{j}}  f(x) dx  \mbox{  for  }  x_{i-1} \le x < x_{i},
$$
where hereafter we define $x_i = i/n$, for $i =1,2, \dots, n$. Thus, moving attention 
to $ ( n M_{\beta, a} M_{\beta_a}^T )^{-1}$  (after introducing the 
appropriate hard-edge scaling),
the action of $ n^{-1/2} M_{\beta, a}^{-1}$  on $L^2[0,1]$  reads
$$
   \Bigl( ( \sqrt{n} M_{\beta, a}  )^{-1}   f \Bigr) (x) =  \sum_{j=1}^{\lfloor nx  \rfloor}  
     \frac{\sqrt{\beta n}}{\chi_{( \lfloor nx  \rfloor +a) \beta }}  \prod_{k=j}^{\lfloor nx  \rfloor - 1}  
     \frac{\tilde{\chi}_{k \beta }}{\chi_{(k+a)\beta}}   \, \int_{x_{j-1}}^{x_{j}} f(x) dx.
$$
In other words,  $n^{-1/2} M_{\beta, a}^{-1}$ is 
equated with the integral operator  $K_{\beta,a}^n$ with (discrete) kernel
\begin{eqnarray}
\label{thekernel}
k_{\beta,a}^n(x,y) &=& \frac{\sqrt{\beta n}}{\chi_{(i+a)\beta}} \exp\left\{\sum_{k=j}^{i-1} \log{
\tilde{\chi}_{k \beta}}-\log{\chi_{(k+a)\beta}} \right\}  \,\one_L(x,y) 
\end{eqnarray}
where  $\one_L  = \one_{\{x_{i-1} \leq x <x_{i}\}}\one_{\{x_{j-1} \leq y <x_{j}\}}$ and $i > j$.

With this set-up, the basic convergence result we need is the following.

\begin{lemma} 
\label{Lawlemma}
There is a Brownian motion $b(\cdot)$ such that
for $x < y$ lying in $(0,1]$  
\begin{equation}
\label{conv1}
   \frac{\sqrt{n \beta}}{ \chi_{( \lfloor nx \rfloor +a) \beta} }   \Rightarrow \frac{1}{\sqrt{x}}
\end{equation}
and
\begin{equation}
\label{conv2}
   \sum_{k = \lfloor ny \rfloor}^{  \lfloor nx \rfloor }  ( \log  \tilde{\chi}_{k \beta}  - \log {\chi_{(k+a)\beta}} )  
    \Rightarrow    (a/2) \log (y/x) +  \int_y^x \frac{db_z}{\sqrt{\beta z}}, 
\end{equation}
in law in the Skorohod topology. 
Morever, the exist tight random constants $\kappa_n>0$ and $\kappa_n' >0$
which are independent of $\beta$ so that
\begin{equation}
\label{tight1}
\sup_{1 \le k \le n}  \frac{ \sqrt{ k \beta }  }{ \chi_{(k+a) \beta} }  \le  \kappa_n
\end{equation}
and, with $T(x) =\frac{1}{\beta} \log \frac{1}{x}$,
\begin{equation}
   \sum_{k=j}^{i-1}   \left( \log \tilde{\chi}_{k \beta}   -  \log  \chi_{(k+a)\beta} \right) 
    -  (a/2) \log (j/i)   \le \kappa_n^{\prime}  (1 + T^{3/4}(x_i)  + T^{3/4} (x_j)) 
\label{tight2}
\end{equation}
for all $1 \le j < i \le n$.
\end{lemma}

The first part of the lemma   ((\ref{conv1}) and (\ref{conv2})  together)
identifies the limiting operator $K_{\beta,a}$.  Namely, for $n \uparrow \infty$ 
it should be that  $k_{\beta,a}^n(x,y)$ approaches
\begin{equation}
\label{limitkernel}
  k_{\beta, a}(x,y) \equiv x^{-\frac{1+a}{2}}  \exp{ \left[  \int_y^x \frac{db_z}{\sqrt{\beta z}} \right]} 
      y^{a/2} \, \one_{y< x}.
\end{equation}      
The second part, or the bounds (\ref{tight1}) and (\ref{tight2}), provide the needed compactness
and more.  As we shall prove:

\begin{lemma} 
\label{HSlemma}
$K_{\beta, a}$ is almost surely Hilbert-Schmidt.  Also,
there exists a probability space on which all $K_{\beta, a}^n$ and $K_{\beta, a}$ are defined,
and such that  any sequence
of  the operators $K_{\beta, a}^n$ contains a subsequence which converges to $K_{\beta, a}$
in Hilbert-Schmidt norm with probability one.  In particular, for whatever $n \uparrow \infty$ we
can find an $n' \uparrow \infty$ along which
$$
  \lim_{n' \uparrow \infty}  \int_0^1 \int_0^1  | k_{\beta,a}^{n'}(x,y)(\omega) - k_{\beta,a}(x,y)(\omega)    |^2 \, dx \,dy = 0   
$$
almost surely.
\end{lemma}

Granted this we may complete the proof of the main result.

\begin{proof}[Proof of Theorem \ref{mainthm}]
Working on the probability space promised in Lemma \ref{HSlemma},
the argument is reduced to a deterministic setting.     Start with the scaled minimal 
$(\beta,a)$-Laguerre eigenvalue 
$$ 
  n \lambda_0(n) =  \inf_{||v||_{\ell_2} = 1}  \langle v, n M_{\beta,a}  M_{\beta, a}^{T} v \rangle  
  =    \Bigl(  \sup_{||f||_{L^2} =1}   \langle f, ( K_{\beta, a}^n)^{T}  K_{\beta, a}^n f \rangle  \Bigr)^{-1}
   =  ||  (K_{\beta, a}^n)^{T}  K_{\beta, a}^n ||^{-1},
$$
where $|| \cdot ||$ is the $L^2 \mapsto L^2$ operator norm, and the final equality holds
simply because 
$(K_{\beta, a}^n)^{T}  K_{\beta, a}^n $ is non-negative symmetric.  Assume for the
moment that, as claimed, $K_{\beta, a}$ is almost surely Hilbert-Schmidt.  Then 
$ K_{\beta, a}^{T}  K_{\beta, a} $ is non-negative symmetric and compact (trace class
even) with a well defined maximal eigenvalue also equal to the norm 
$|| K_{\beta, a}^T K_{\beta, a}  ||$.
For short, we notate $   \hat{\Lambda}_0 >
\hat{\Lambda}_1 > \cdots $ the eigenvalues of $ K_{\beta, a}^{T}  K_{\beta, a} \equiv B$, and
similarly write $\{ \hat{\Lambda}_k^n \}$  for the (decreasing)  eigenvalues of $ (K_{\beta, a}^{n})^{T}  K_{\beta, a}^{n} \equiv B_n$. 
The simplicity of the limiting eigenvalues is also assumed here; it will be deduced from the differential form
of the eigenvalue problem introduced in the next Section, see  {\ref{eq:PsiSystem}} and the surrounding
discussion.

Next,  for whatever  sequence $n \uparrow \infty$, Lemma \ref{HSlemma}  
allows a choice of subsequence along which 
$ || K_{\beta, a}^{n'} - K_{\beta, a} ||_{HS}  \rightarrow 0$.   The same holds for the transposes,
and hence $ B_{n'} $ converges strongly to 
$  B $.  It follows that the norms themselves converge along this subsequence:
$  \hat{\Lambda}_0^n  =  || B_{n'} || \rightarrow || B ||  =  \hat{\Lambda}_0 $ with probability
one.  But this is to say that for any sequence of the original eigenvalues $ n \lambda_0(n)$,
there  exists a subsequence along which  these points converge almost surely to 
$1/ \hat{\Lambda}_0 $.  That is of course equivalent to the full convergence
statement.
    
As to $ n \lambda_1, n \lambda_2,...$,  we first show that the convergence 
(along perhaps a further 
subsequence) of the ground state eigenvectors
is a by-product of the above.     Define $\{f_n\}$ and $f$ in  $L^2[0,1]$ with unit norm  by
$$
   \langle f_n , B_n f_n \rangle =   \hat{\Lambda}_0^n,  \   \   \   
    \langle f , B f \rangle =   \hat{\Lambda}_0.
$$
Remaining in the introduced setting, we have $|| B_n f_n ||_{L^2} \rightarrow || B f ||_{L^2}$.
Also,  being uniformly bounded
in $L^2$,  $f_n$ has a weakly convergent subsequence:  $f_{n'}  \rightharpoonup
f_{\infty}$.  Then, for  any $\phi \in L^2$, 
$$
   \langle{ \phi, B_n f_n  - B f_{\infty} \rangle}  =  \langle{ \phi, (B_n - B) f_n \rangle}  + 
            \langle{ B \phi,   f_n - f_{\infty} \rangle}  
$$
tends to zero (the first term by norm convergence, the second by boundedness of $B$).   
Having  weak convergence ($B_n f_n \rightharpoonup B f_{\infty}$)
 plus convergence of its norm, we conclude there
is a strongly convergent  subsequence of $\{B_n f_n\}$.
Coupled with $ B_n f_n  = \hat{\Lambda}_0^n f_n$  
and $\hat{\Lambda}_0^n \rightarrow \hat{\Lambda}_0$, this  implies a strongly convergent 
subsequence for the $\{f_n\}$ themselves, which by continuity can only wind up at $f$. 

Finally,  place yourself along this sequence where $|| f_n - f ||_{L^2} \rightarrow 0$, and 
denote
by $P_{f_n}$ the projection onto the orthogonal complement of $f_n$ in $L^2$.  At once 
we find that $ P_{f_n} B_n P_{f_n} $ converges strongly to $ P_f B P_f$ (with obvious
notation), and so  $\hat{\Lambda}_1^n =  ||   P_{f_n} B_n P_{f_n} || \rightarrow || P_f B P_f ||
= \hat{\Lambda}_1$.  The implication for $n \lambda_1$ is clear, and an induction argument
extends the picture to the almost sure convergence of any finite number of  Laguerre eigenvalues.

Reflect upon the fact that we have proved convergence (in law) of say $ \{ (n \lambda_k)^{-1}  \}$
for $k = 0, \dots, m$ to  the top $m$ eigenvalues of the integral operator 
$B = K_{\beta, a}^T K_{\beta,a}^{}$ which we now write out.  In particular, its spectral problem reads
\begin{eqnarray}
\label{fullfirstop}
   f(x) & =  & \lambda   \int_x^1 x^{a/2} e^{\int_x^y \frac{db_s}{\sqrt{\beta s}}  }  y^{-(a+1)}  \int_0^y  
    e^{\int_z^y  \frac{db_s}{\sqrt{\beta s}}}  z^{a/2}   \, f(z) \,  dz \,   dy   \\
         &  =  &   \lambda \int_0^1 (x y)^{a/2}  \left( \int_{x \vee y}^1 e^{-2 \int_z^1 \frac{db_s}{\sqrt{\beta s}}  } 
                                                \, z^{-(a+1)}  dz \right)    e^{\int_x^1 \frac{db_s}{\sqrt{\beta s}}  } \,  e^{\int_y^1 \frac{db_s}{\sqrt{\beta s}}  } f(y) \, dy, \nonumber
\end{eqnarray}
after an integration by parts.  Again, we seek here an $f \in L^2[0,1]$, which inherits the continuity
and vanishing of the kernel at $x=1$.

To recover the advertised limit operator (\ref{Greensfunction}), make the substitution
$g(x) $ $=$  $x^{-a/2} e^{\int_x^1 \frac{db_s}{\sqrt{\beta s}}  }  f(x) $ in conjunction with the time change $ \int_x^1 
s^{-1/2} {db_s}  =   \hat{b}(\log (1/x))$  with a new 
Brownian motion $\hat{b}$  to express (\ref{fullfirstop}) in the equivalent way:
$$
  g(x) = \lambda \int_0^1   \left(  \int_{x \vee y}^1  e^{- \frac{2}{\sqrt{\beta}}  \hat{b}(\log 1/z)} 
           z^{-(a+1)} \,  dz \right)  g(y)  \, y^a  e^{ \frac{2}{\sqrt{\beta}}  \hat{b}(\log 1/y) }  \, dy.
$$
With $f \in L^2[0,1]$, $g$ resides in $ L^2([0,1], {m})$ for $m(dx) = 
x^{a} e^{ \frac{2}{\sqrt{\beta}}  \hat{b}(\log x^{-1}) } \, dx$.  Last,  the change of variables 
$(x,y) \mapsto (e^{-x}, e^{-y})$ will produce  the form of (\ref{Greensfunction}) quite exactly, along
with transforming the Dirichlet condition at one ($f(1) = g(1)= 0$) into that at the origin $(\psi(0) \equiv
(g \circ \exp)(0) = 0$).
\end{proof}

\subsection*{Estimates}

Before establishing Lemmas \ref{Lawlemma} and \ref{HSlemma} we make the simple observation:

\begin{proposition} 
\label{dominating}
For any constant $C$ and $a > -1$, the integral operator on $L^2[0,1]$ with kernel
$$
   k_C(x,y) =  C \exp{\Bigl[  C (\log(1/x))^{3/4} +  C (\log(1/y))^{3/4}  \Bigr]} \,   \frac{y^{a/2}}{x^{(a+1)/2}} \,  \one_{ y < x} 
$$
is Hilbert-Schmidt. 
\end{proposition} 

\begin{proof} The change of variables $x = e^{-s}$ and $y = e^{-t}$ employed just above produces
$$
  \int_0^1 \int_0^1 | k_C(x,y)|^2 \, dx dy =  C^2 \int_0^{\infty}   e^{ 2 C s^{3/4} + a s }  \int_s^{\infty} e^{ 2 C t^{3/4} - (a+1) t} \,  dt \,  ds , 
$$
and the latter is clearly finite if (and only if) $a > -1$.
\end{proof}

\begin{proof}[Proof of Lemma \ref{HSlemma}]  Now taking Lemma \ref{Lawlemma} for granted, we can find a subsequence over which
we have the joint convergence in law,
\begin{eqnarray}
\label{jointly} 
      \frac{\sqrt{n \beta}}{ \chi_{(\lfloor nx \rfloor+a) \beta} }  &  \Rightarrow  & \left( \frac{1}{\sqrt{x}}, \,  0 < x \le 1 \right),  \nonumber  \\ 
   \sum_{k = \lfloor ny \rfloor}^{  \lfloor nx \rfloor }  ( \log  \tilde{\chi}_{k \beta}  - \log {\chi_{(k+a)\beta}} )  
      &  \Rightarrow  &    \left( (a/2) \log (y/x) +  \int_y^x \frac{db_z}{\sqrt{\beta z}}, \,  0 < y \le x  < 1 \right),  \\
      \kappa_n^{}, \kappa_n^{\prime}  & \Rightarrow & \kappa, \kappa^{\prime}.  \nonumber   
\end{eqnarray}
Then, Skorohod's representation theorem (Theorem 1.8, Chapter 2 of \cite{EK}) furnishes a probability space on which 
each of the above occurs with probability one.  The first two items of (\ref{jointly}) take place a.e. in  $(0,1]$, and so on this
new space it holds
\begin{equation}
\label{strongpath}
   P \Bigl(  \lim_{n \uparrow \infty} k_{\beta,a}^n(x,y)(\omega) = k_{\beta,a}(x,y)(\omega)  \mbox{ for }  a.e. \, x, y \in [0,1]^2 \Bigr) = 1.
\end{equation}
That
$$
    \int_0^1 \int_0^1 | k_{\beta,a}^n(x,y)(\omega) - k_{\beta,a}(x,y)(\omega)  |^2 \, dx dy    \rightarrow 0  \ \  a.s.
$$
will follow if we can supply an a.s.\ finite constant $C(\omega)$ such
\begin{equation}
\label{doms} 
   \sup_{n > 0} k_{\beta,a}^n(x,y)(\omega) \le k_{C(\omega)}(x,y)  \  \mbox{ and }  \    k_{\beta,a}(x,y)(\omega) \le k_{C(\omega)}(x,y) 
\end{equation}
for almost all $x,y \in [0,1]$ and $\omega$.

Again by Lemma \ref{Lawlemma}, for each $n$ it holds
\begin{equation}
\label{knbound}
  k_{\beta,a}^n(x,y)(w) \le  \kappa_n(\omega) x_i^{-(a+1)/2} y_j^{a/2}  \exp{ \left[  \kappa_n^{\prime}(\omega) ( 1 + T^{3/4}(x_i) + T^{3/4}(y_j) ) \right] } 
\end{equation}
where $x \in [x_i. x_{i+1}), y \in [y_j, y_{j+1})$.  But now we are allowed to assume that both $\kappa_n$ and $\kappa_n^{\prime}$ converge, 
and thus are bounded almost surely by say $ 2 ( \kappa \vee \kappa')$ for sufficiently large $n$. The continuity of the functions $x \mapsto T(x)$
and $ x \mapsto x^{p}$ (on $(0,1]$) then enables us to fit the right hand side of $(\ref{knbound})$ under a fixed  $k_{C}$ independently of $n$.
For the limit kernel,  $k_{\beta,a}(x,y)(\omega)$ simply note that its exponent could have been expressed from the start as 
$$
     \int_y^x \frac{db_z}{\sqrt{\beta z}}  =  \frac{1}{\sqrt{\beta}}  \Bigl[  \tilde{b} ( \log (1/x) ) -  \tilde{b} ( \log(1/y) ) \Bigr].
$$
(The equality is in law with a different Brownian motion living on the same probability space.)  By the law of the iterated logarithm,
$ \tilde{b}(a) \le c(\omega) (1 + [ a \log \log(1+ a)]^{1/2}) $ for a  random $c(\omega)$ and all $a > 0$, and certainly $[a \log \log(1+ a)]^{1/2}
\le c' a^{3/4}$ with a (non-random) $c'$ and all $a$ large enough.  Thus, the second half of (\ref{doms}) holds with $C(\omega) \le
 \beta^{-1/2}  c' \, c(\omega)$,
and the proof is complete.

Note here we immediately passed to a fixed subsequence and then chose a favorable probability space, while the statement of the lemma
was worded with the convergence of the $(\beta, a)$-Laguerre eigenvalues in mind. That is,  build all $k_{\beta,a}^n$, each tied to
a   $(\beta, a)$-Laguerre eigenvalue, on the same space as $k_{\beta, a}$ and then note for whatever $n \uparrow \infty$ there is a subsequence
along which everything above holds.  Either way the upshot is the same.
\end{proof}

Turning to the proof of Lemma \ref{Lawlemma}, we record without proof the following facts.

\begin{proposition}
\label{moments}
For $\chi_r$ a chi random variable of index $r > 0$, 
\begin{equation}
\label{chimoment}
  E [ \chi_r^p ] = 2^{p} \frac{\Gamma\left(\frac{r+p}{2}\right)}{\Gamma(r/2)}
\end{equation}
for any $p > - r$. Also, as $r \rightarrow \infty$,
\begin{eqnarray}
\label{logmoment}
  E[ \log \chi_r ] = \frac{1}{2} \log r - \frac{3}{2r}  + O(1/r^2), \  \  Var[ \log \chi_r ] = \frac{1}{2r} + O(1/r^2) , 
\end{eqnarray}
while $   E [  (\log \chi_r - E \log \chi_r )^{2m} ]  = O(1/r^{m})$ 
for  positive integer $m$.
\end{proposition}

\begin{proposition}[After Theorem 1.3, Chapter 7 of \cite{EK}]
\label{convprop}
 Let $y_{n,k}$ be a sequence of mean-zero processes starting at $0$
with independent increments $\Delta y_{n,k}$.  Assume, 
\begin{equation}
\label{convassumps}
   n E  (\Delta y_{n,k})^2  = f(k/n) + o(1), \  \  n E  (\Delta y_{n,k})^4  = o(1)
\end{equation}
uniformly for $k/n$ in compact sets of $[0, T)$ with a continuous $f \in L_{loc}^1[0,T)$.  Then $y_n(t) = y_{n, [nt]} \Rightarrow \int_0^t {f^{1/2}(s)} db(s)$
with a standard Brownian motion $b$ (in the Skorohod topology).   
\end{proposition}

\begin{proof}[Proof of Lemma \ref{Lawlemma}]
Start with (\ref{conv2}).
By the first estimate of (\ref{logmoment}),
$$
  \lim_{n \rightarrow \infty} \sum_{k= \lfloor ny \rfloor}^{ \lfloor nx \rfloor } ( E \log \tilde{\chi}_{k\beta}  - E \log \chi_{(k+a) \beta} )  = \frac{a}{2} \log(y/x)
$$
uniformly for $y < x$ restricted to compact sets of $(0,1]$.  Thus, for (\ref{conv2}) it is enough to demonstrate the weak convergence
\begin{equation}
\label{newconv2}
     \sum_{k = [nx]}^n  ( \log \chi_{(k+c) \beta} - E \log \chi_{(k+c) \beta} ) \Rightarrow    \int_x^1 ( 2 \beta z)^{-1/2} \, {db(z)}
\end{equation}
where $c$ is any fixed number.  Indeed, the exponent of the discrete kernel is comprised of two such independent sums,
and the promised limit will follow as $b_1 + b_2 = \sqrt{2} b_3$ in law for independent Brownian motions $b_1, b_2, b_3$.
Now refer to Proposition \ref{convprop} and  view the processes on the left of 
(\ref{newconv2}) as starting from  $0$ at $x = 1$ and evolving toward $x = 0$ (or take $t = 1-x$ in the proposition). Then,
the second estimate of (\ref{logmoment}) yields the first part of (\ref{convassumps}) with $f(t) = 1/(2\beta t)$; the estimate right after  (\ref{logmoment})
with $m=2$ produces the second half of (\ref{convassumps}) as $x$ is always $> 0$.  This finishes the job.

The convergence (\ref{conv1}) is easier. For any fixed $x \in (0,1]$, it is just an instance of the law of large  numbers.  The tightness required  to ensure
process level convergence is  also elementary: via (\ref{chimoment}) one can obtain the increment bound 
$$ 
   E  \Bigl(  \sqrt{\frac{(r+1)\beta}{\chi_{(r+a+1)\beta}^2}}  -  \sqrt{\frac{r\beta}{\chi_{(r+a)\beta}^2} }  \, \Bigr)^2 = O(1/r^2)
$$
which more than suffices.  While here we  dispense of  (\ref{tight1}).  
First use the sum bound,
$$
  P \left(  \sup_{1 \le k \le n} \frac{ \sqrt{k \beta}}{\chi_{(k+a)\beta}} > M \right) \le \sum_{k=1}^{n} P \left(   \frac{\chi_{(k+a)\beta}}{\sqrt{k \beta}}     < \frac{1}{M} \right).
$$
Then, 
 employing the explicit density $P(\chi_r \in ds) =  \frac{{2}^{1-r/2}}{\Gamma(r/2)} s^{r-1} e^{-s^2/2} \, ds$, one can perform a  Laplace-type
 estimate to find the $k$-th term on the right hand side is upper bounded by $C (\sqrt{e}/M)^k$ with $C$ depending only on $\beta$.  Since
 $\sum_{k=1}^{\infty} (\sqrt{e}/M)^k$ may be made arbitrarily small by choice of $M$, 
 the desired tightness of the  random variables 
 $\sup_{k\le n} (  \sqrt{k \beta} / \chi_{(k+a)\beta} )$ follows.

 The final piece, or (\ref{tight2}), is the most elaborate but really 
 comes down to reworking the standard proof of the upper bound in the law of the iterated logarithm.
 Define, 
 $$
    A_{x}^n = \sum_{k=j}^{n-1} ( \log \tilde{\chi}_{k \beta} - \log \chi_{(k+a) \beta} ) - \frac{a}{2} \log( j/n)
 $$
 for $x \in [x_j, x_{j+1})$, and $h(x) = [ 2x \log \log x]^{1/2}$.  We will in fact show that
 \begin{eqnarray}
 \label{lastgoal}
     \sup_{1 \le j \le n-1}  \Bigl(  ( A_{x_j}^n \vee 0) / h(T(x_j)) \Bigr)  \mbox{  are tight in distribution,} 
 \end{eqnarray}
 where again  $T(x) = \frac{1}{\beta} \log \frac{1}{x}$.
This is stronger than what is claimed.  
 
 Set
 $$
    Y_j^n \equiv \exp ( A_{x_j}^n ) =  \prod_{k=j}^{n-1}  \frac{\tilde{\chi}_{k \beta}}{\chi_{(k+a)\beta}}   \left( \frac{k+1}{k} \right)^{a/2},
    \  \mbox{ and }  Z_j^n  \equiv  ( Y_j^n )^{\lambda}  E [ ( Y_j^n )^{\lambda} ]^{-1}
 $$
 with a small positive $\lambda$ (the precise conditions on $\lambda$ follow shortly).  
 The sequence $j \mapsto
 Z_{n-j}^n$ is a martingale for  $j=1,2,\dots$ with $E[ Z_j^n] = 1$ for all $j$. Hence, by Doob's inequality  
$$
  P \Bigl( \max_{ \ell \le j \le n-1} Z_j^n  \ge e^{\lambda b}  \Bigr) \le e^{-\lambda b},
$$
or  
\begin{equation}
\label{s1}
   P \Bigl( \max_{\ell \le j \le n-1} ( \lambda A_{x_j}^n -  \log E  [  \exp ( \lambda A_{x_j}^n ) ] )  \ge  b \Bigr) \le e^{-\lambda b}
\end{equation}
for $b > 0$.  

For the next move we need an estimate on the moment generating functions of $A_{x_j}^n$, the proof
of which we will return to at the end of the section.

\begin{claim} 
\label{mgf}
For all  $\lambda> 0$ sufficiently small $(\lambda < (\beta/2) [ (1+a) \wedge 1]$ will do$)$,
\begin{equation}
\label{eqmgf}
         E\Bigl[  e^{\lambda A_{x_j}^n} \Bigr] = \exp \Bigl\{ \frac{\lambda^2}{2\beta} \log (1/x_j) + \Theta_n(j) \Bigr\}
\end{equation}
with $|\Theta_n(j)| \le  C$ for constant $C = C(a,\beta)$.
\end{claim}

Using (\ref{eqmgf}) in (\ref{s1}), we have 
$$
   P \Bigl( \sup_{x_{\ell}  \le t < 1} \left\{  A_{t}^n -   \frac{\lambda}{2 \beta}  \log(1/t) +  \frac{1}{\lambda} \Theta_n(nt) \right\}  \ge  b/\lambda \Bigr) \le e^{-\lambda b}
$$
with $\Theta_n(t)$ understood via interpolation.  
Now choose $\theta > 1$, a positive constant $M$ and set $\lambda = M \theta^{-m} h(\theta^m)$,
$b = M h(\theta^m)/2$. (To choose $M$ large one  must take $\theta$ large as well to respect
the condition on $\lambda$ set down in Claim \ref{mgf}.)  The previous display will then
imply
\begin{equation}
\label{s2}
  P \left(  \sup_{\theta^m < T(t) < \theta^{m+1} }  A_t^n  \ge (M+1) h(\theta^m)   \right) \le (m \log \theta)^{-M^2}.
\end{equation}
Here we have used the uniform bound on $\Theta_n(t)$ to fit $\lambda^{-1} \Theta_n(t)$ under 
$h(\theta^m)$ by choice of $\theta$ and so $M$. In particular, $\lambda^{-1}$ $= M^{-1} \theta^m h^{-1}(\theta^m)$
$\le M^{-1} h(\theta^m)$ if $\log \log \theta > 1/2$.

Finally return to the goal (\ref{lastgoal}), re-expressed as seeking a bound of type 
$$
 P \left( \sup_{0 < t < 1} [A_t^n]^{+}/ h(T(t))  > N \right) \le \varepsilon(N) \mbox{  where } 
 \varepsilon(N) \downarrow 0 \mbox{ as } N \uparrow \infty.
$$
Note in addition that the supremum inside the probability over any truncated range $x < t < 1$ (for $x >0$, rather
than $0<t<1$) poses no problem.  Indeed, the process $ t \mapsto A_t^n$ has already  been shown to be convergent
in that regime.  On the other hand, the troublesome tail is bounded by
$$
 \sum_{m=1}^{\infty}  P \left(  \sup_{\theta^m < T(t) < \theta^{m+1} }  A_t^n / h(T(t))  \ge N  \right) \le 
 \sum_{m=1}^{\infty} ( m \log \theta)^{-(N-1)^2}
$$
with the aid of (\ref{s2}), completing the proof.
\end{proof}

\begin{proof}[Proof of Claim \ref{mgf}] By (\ref{chimoment}), the left hand side of (\ref{eqmgf}) equals
$$
    \prod_{k=j}^{n-1}  \frac{ \Gamma \left( \frac{k\beta + \lambda}{2} \right)}{ \Gamma \left( \frac{k \beta}{2} \right) } \frac{ \Gamma \left( \frac{(k +a) \beta - \lambda}{2} \right)}{ \Gamma \left( \frac{(k+a) \beta}{2} \right) }   \left( \frac{k+1}{k} \right)^{\lambda a/2}.
$$
Taking logarithms, we must estimate the sum $\sum_{k=j}^{n-1} s_k$ where
\begin{eqnarray}
\label{term}
  s_k &  =  & 
  \log  \Gamma \left( \frac{k\beta + \lambda}{2} \right) - \log  \Gamma \left( \frac{k \beta}{2} \right)   + 
  \log  \Gamma \left( \frac{(k +a) \beta - \lambda}{2} \right) - \log \Gamma \left( \frac{(k+a) \beta}{2} \right) 
    \nonumber \\      
  &  &  + \frac{\lambda a}{2} \log \left(1 +  \frac{1}{k} \right). 
\end{eqnarray}  

Introduce Stirling's approximation in the form  
$$
  \left| \log \Gamma(z) - \left(z-\frac{1}{2}\right) \log{z} +z - \frac{\log{2\pi}}{2} - \frac{1}{12 z} \right| \leq \frac{c}{z^{2}}.
$$
The $O(z^{-2})$ error term produces a constant multiple of $k^{-2}$ when applied in $(\ref{term})$. Differences
such as $ (k \beta/2)^{-1} - ((k \beta -\lambda)/2)^{-1}$ and the like stemming from the $1/(12z)$ terms are similarly bounded.  When summed, both contributions
produce constants which are then absorbed into the $\Theta_n$.  Also, the constant and $z$-terms obviously
cancel throughout the log-gamma expressions when the above estimate is applied in (\ref{term}). 

Move to the terms of type $(z-1/2)\log z$. A bit of algebra will lead to
\begin{eqnarray*}
  s_k &  =  &  \frac{k\beta -1}{2} \log \left[ \left(1 + \frac{\lambda}{k \beta} \right) \left( 1 - \frac{\lambda}{(k+a)\beta} \right)
    \right]
                   - \frac{\lambda}{2} \log \left[  \left(1 + \frac{\lambda}{k \beta} \right) / \left( 1 - \frac{\lambda}{(k+a) \beta} \right)
                 \right] \\      
  &  &  + \frac{a \beta}{2} \log \left(1 - \frac{\lambda}{(k+a) \beta} \right) - \frac{\lambda}{2} \log \left(1 + \frac{a}{k} \right)
      - \frac{\lambda a}{2} \log \left(1 +  \frac{1}{k} \right)  + O \left( \frac{1}{k^2} \right).  
\end{eqnarray*}
Since $|\log(1+s) - s| \le s^2$ for $s >-1/2$, we conclude $s_k = \frac{\lambda^2}{2 \beta k} + O(k^{-2})$, which 
establishes the claim upon summation from $j$ to $n-1$.
\end{proof} 

\section{Riccati map and a second diffusion}

Riccati's substitution  takes a linear second order operator into one of first order, at the price of introducing
a quadratic nonlinearity.  Its use in the study of random spectra has a long history, dating
back to Halperin \cite{Hal}.  To employ it here we must first recover the differential form of the
eigenvalue problem from the established integrated version 
$\psi = \lambda \mathfrak{G}_{\beta,a}^{-1} \psi $, which reads in full:
\begin{eqnarray*}
    \psi(x) & = & \lambda \int_0^{\infty}  \int_0^{x \wedge y}  s(dz) \,    \psi(y) \, m(dy)  \\
                & = & \lambda \int_0^{\infty}   \left( \int_0^{ x \wedge y} \exp{[ a z + \stb b(z) ]} \, dz  \right)  \psi(y)   \exp{[ -(a+1) y - \stb b(y)]} \, dy.  
\end{eqnarray*}
Noting that any $f \in L^2[ m] $ is also in $L^1[m]$,  $  \mathfrak{G}_{\beta,a}^{-1} f$  is easily seen to be differentiable after writing the right hand side
as separate terms.  This property
is inherited by $\psi$, and we compute
\[
\psi'(x) = \lambda \exp[ a x + \stb b(x)  ]  \int_x^{\infty} \psi(y) \exp [ - (a+1) y  - \stb b(y)] \, dy,
\]
to find that $\psi$ is actually in $C^{3/2-}$.  Continue by taking (It\^o) differentials to arrive at the system
\begin{eqnarray}
  d \psi^{\prime}(x) & =  & \stb \psi^{\prime}(x) db(x) +  \left( (a+\tb ) \psi^{\prime}(x) - \lambda e^{-x} \psi(x)  \right) dx , \nonumber \\ 
   d \psi(x) &=  & \psi^{\prime}(x) dx,
\label{eq:PsiSystem}
\end{eqnarray} 
which is the appropriate way to interpret 
$\mathfrak{G}_{\beta,a} \psi = \lambda \psi$.
Taken independently of the preceding developments, (\ref{eq:PsiSystem}) has globally Lipschitz coefficients of linear growth, and as such defines
(for fixed $\lambda$) a unique Markov process $x \mapsto (\psi(x), \psi^{\prime}(x) )$  for any specified $(\psi(0), \psi^{\prime}(0))$ pair. 
As the uniqueness is pathwise we conclude along the way the simplicity of the corresponding eigenvalues: for given $\lambda$, any two $L^2$ solutions 
of $\psi = \lambda \mathfrak{G}_{\beta,a}^{-1} \psi $ vanishing at the origin must be constant multiples of one another.

Now bring in Riccati's map, $p(x) = \psi'(x)/\psi(x)$, valid away from the zeros of $\psi$. 
Since $\psi$ is continuously differentiable, we find from  (\ref{eq:PsiSystem}) and elementary calculus:
\begin{equation}
dp(x) = \stb p(x)\, db(x) + \left( (a+ \tb ) p(x)- p^2(x) - \lambda e^{-x}  \right) dx,
\label{Riccati}
\end{equation}
defining yet another Markov process for any fixed $\lambda$.
The relevance of (\ref{Riccati})  in counting eigenvalues of $\mathfrak{G}_{\beta,a}$
is first understood through the truncated operator ${\mathfrak{G}}_{\beta, a}^L$,
indicating ${\mathfrak{G}}_{\beta, a}$ restricted to $[0,L]$ with Dirichlet conditions at both endpoints.

\begin{lemma}
\label{countlemma}
Consider the unique diffusion $p(x) = p(x;   \lambda )$ 
started at $+\infty$ at $x=0$, and restarted at $+\infty$
immediately after any passage to $-\infty$.  The number of eigenvalues of
 ${\mathfrak{G}}_{\beta, a}^L$ less than $\lambda$ is equal in law to the number of explosions  of
 $p$ before
$x=L$.
\end{lemma}

\begin{proof}  This is well understood, and our treatment here is much the same as in \cite{RRV}, Section 3.
Take the sine-like solution of (\ref{eq:PsiSystem}), that is,  $\psi_0(x,\lambda)$  subject to $\psi_0(0, \lambda) =0$
and $\psi_0^{\prime}(0, \lambda) = 1$. Plainly, $\Lambda$ is an eigenvalue of $\mathfrak{G}_{\beta, a}^L$ if only if
$\psi_0(L, \Lambda) = 0$.  Regarding the ground state eigenvalue $\Lambda_0(L)$: if for any $\lambda$ $\psi_0(x, \lambda) > 0$ 
for $0 < x \le L$, then it must be that $\Lambda_0(L) > \lambda$, as an examination  of (\ref{eq:PsiSystem}) shows.  That is,
the event that $\{\Lambda_0(L)> \lambda \}$ is equal in law to the event
 $\{x \mapsto \psi_0(x,\lambda) \mbox{ has no roots before } x=L \}$. Continuing, additional zeros of 
 the (almost surely continuous) function $\lambda \mapsto \psi_0(L, \lambda)$ (and
 so additional eigenvalues) only occur by increasing $\lambda$, whereupon all other roots (in the $x$-variable) move to 
 the left. This equates the event that the $k$-th eigenvalue of $\mathfrak{G}_{\beta, a}^L$ lies above a fixed $\lambda$ and the event that 
 $\psi_0(x,\lambda)$
 has at most $k-1$ roots on $(0,L)$.

Now move to the $p(x,\lambda)$ formed from $\psi_0(x,\lambda)$ and its derivative. 
By appealing again uniqueness of solutions to (\ref{eq:PsiSystem}),  note $\psi_0$ and $\psi_0'$ cannot vanish simultaneously.
(In particular, the zeros of $\psi_0$ are isolated, and must be either finite in number or form a sequence tending to infinity.) Thus, at any root
$\mathfrak{m}$ of $x \mapsto \psi_0(x, \lambda)$, including $\mathfrak{m}=0$, an examination of signs shows that 
$\lim_{\varepsilon \downarrow 0} p(\mathfrak{m} + \varepsilon, \lambda) = + \infty$ and, when $\mathfrak{m} >0$,
$\lim_{\varepsilon \downarrow 0} p(\mathfrak{m} - \varepsilon, \lambda) = - \infty$.  That is, counting roots of $\psi_0(\cdot, \lambda)$ is to count passages of
the corresponding $p(\cdot, \lambda)$ to $-\infty$, after subsequent re-starts at $+\infty$.

To see that the $p$-picture stands on its own is to show that there is a unique solution of (\ref{Riccati}) starting from $+\infty$. Replacing 
the $-\lambda e^{-x}$ term in the drift with any negative constant produces a homogeneous motion with an entrance boundary at $+\infty$
(and which hits $-\infty$ with probability one).  This process (begun at $+\infty$) may be constructed unambiguously via speed and scale, see
again \cite{IM}. By successive dominations of the inhomogeneous $p$ in the statement  by such homogeneous versions over all short times,
one may conclude the existence and uniqueness of the former.
\end{proof}

Theorem \ref{DiffThm} now follows by taking $L \rightarrow \infty$ in Lemma  \ref{countlemma} with the
aid of the next fact.

\begin{lemma} 
As $L \rightarrow \infty$, the top $k$ eigenvalues of $ \mathfrak{G}_{\beta, a}^L $ converge to the top $k$ eigenvalues of $\mathfrak{G}_{\beta, a}$
with probability one.
\end{lemma}

\begin{proof}  
This again demonstrates the advantage of having explicit inverses. 
Now $(\mathfrak{G}_{\beta, a}^L)^{-1} $  acts on $L^{2}( [0,L], m)$ via
$$
    \left( ( \mathfrak{G}_{\beta, a}^L)^{-1} f \right)(x) = \int_0^{\infty} s_L(x,y) f(y) \, m(dy)
$$
where
$$
   s_L(x,y) = \left[ \int_0^{x \wedge y} s(dz) \right]  \,  \times  \left[ \frac{ \int_{x \vee y}^L s(dz) }{ \int_0^L s(dz) }  \right]  \one_{\{x,y \in [0,L]\}}.
$$
Plainly, $s_L(x,y) \le \int_0^{x \wedge y} s(dz)$ and $ \lim_{L \rightarrow \infty} s_L(x, y) =  \int_0^{x \wedge y} s(dz)$ pointwise in $x$ and $y$, almost surely. 
By dominated convergence we have in the same mode that
$$
  \int_0^{\infty}  \int_0^{\infty} f(x) s_L(x,y) g(y) \,  m(dx) m(dy) \rightarrow   \int_0^{\infty}  \int_0^{\infty} f(x)  \left( \int_0^{x \wedge y} s(dz) \right) g(y) \,  m(dx) m(dy) 
$$ 
for all $f, g \in L^2[\RR_+, m]$, and
$$
 \mbox{ tr} \, (\mathfrak{G}_{\beta, a}^L)^{-1} = \int_0^L s_L(x,x) m(dx)  \rightarrow  \int_0^{\infty}  \int_0^x s(dy)  m(dx) = \mbox{ tr} \, \mathfrak{G}_{\beta, a}^{-1}.
$$
But these last two items imply convergence of $\mathfrak{G}_{\beta, a}^L$ to  $\mathfrak{G}_{\beta, a}$
in trace  norm (see \cite{Simon}, Theorem 2.20); the convergence of the eigenvalues then stems from
the same style of argument used in the proof of Theorem \ref{mainthm}.
\end{proof}

\section{The Hard-to-Soft transition}

Borodin-Forrester \cite{BF} discovered a transition between the hard and soft edge distributions
at $\beta = 1,2,$ and $4$.  Their proof rests on the explicit Fredholm determinant 
or Fredholm pfaffian form of these laws.  For example, at $\beta = 2$ one has that
\begin{equation}
\label{Fred1}
  P ( \Lambda_0( 2,a) > \lambda ) = 1 + \sum_{k=1}^{\infty} \frac{(-1)^k}{k!}  
    \int_0^{\lambda} dx_1 \cdots \int_0^{\lambda} d x_k  \det \Bigl[  K_{Bessel}(x_i, x_j) \Bigr]_{i,j = 1,\dots, k},
\end{equation}
while
\begin{equation}
\label{Fred2}
 P(TW_2 < \lambda) =  1 + \sum_{k=1}^{\infty} \frac{(-1)^k}{k!}  
    \int_{\lambda}^{\infty} dx_1 \cdots \int_{\lambda}^{\infty} d x_k  \det \Bigl[  K_{Airy}(x_i, x_j) \Bigr]_{i,j = 1,\dots, k}.
\end{equation}
Here, 
$$
   K_{Bessel}(x,y) := \frac{ J_a( \sqrt{x}) \sqrt{y} J_a^{\prime}(\sqrt{y})  - \sqrt{x} J_a^{\prime}(\sqrt{x}) J_a(\sqrt{y}) }{x-y}
$$
with $J_a$ the usual Bessel function of the first kind, which is replaced by the Airy function in
$$
  K_{Airy}(x,y) := \frac{ \Ai(x) \Ai^{\prime}(y) - \Ai^{\prime}(x) \Ai(y)}{ x-y}.
$$
For $\beta = 1$ or $4$ the determinants in (\ref{Fred1}) and (\ref{Fred2}) are replaced 
by quaternion determinants (or, equivalently,  pfaffians), but are  comprised of the same class of functions.  
Further, it is a fact that, suitably scaled, $J_a$ goes over into the Airy function as 
$a \ra \infty$, and the analysis of \cite{BF} demonstrates that one may pass this limit inside
the various multiple integrals in (\ref{Fred1}) and its analogues.

By a much different method, employing the Riccati correspondence, we show the same type of phenomena holds at all $\beta > 0$.

From Theorem \ref{DiffThm}, the event that $\{ \Lambda_0(\beta,a) > \lambda \}$
is equivalent in law to the process
$$
   dp(x) = \stb p(x) db(x) + \left( (a  + \tb) p(x) - p^2(x) -  {\lambda} e^{-x} \right) dx 
 $$
never hitting $-\infty$. While from \cite{RRV} we know that
 the probability of the event  $\{ TW_{\beta} < \mu \}$ equals the chance that
a separate motion $q$ given by
\begin{equation}
\label{q_eq}
   dq(x) = \stb  db(x)  + ( x + \mu - q^2(x)) dx
\end{equation}
also never hits $-\infty$.  (Both processes  are begun at $+\infty$.) 
The question is then:  with the scalings
$$
   a = 2 \sqrt{\eta} - \frac{2}{\beta}   > -1 \  \     \mbox{ and } \  \   \lambda =  \eta - \eta^{2/3} \mu ,
$$
does the chance of $p$-explosion go over into that of a $q$-explosion for large $\eta$?

To understand the mechanism, set $\mu = 0$ for a moment.
This scaled  $p$ solves
$$
   dp(x) =  \stb p(x) db(x) + (2 \sqrt{\eta}  p(x) - p^2(x) -  \eta e^{-x} ) dx,
$$
and obviously $\mathfrak{p} =   p/\sqrt{\eta}$ explodes or not with $p$ while satisfying
$$
   d \mathfrak{p}(x)  = \stb \mathfrak{p}(x) db(x) +  \sqrt{\eta} \,
   ( 2 \mathfrak{p}(x) - \mathfrak{p}^2(x) -   e^{-x} ) dx.
$$
For $\eta \uparrow \infty$, $\mathfrak{p}$ comes quickly to the place $\mathfrak{p} = 1$, and, if it manages to
tunnel through this point in a short time, explosion is hard to avoid.  
Within this excursion from $1^{+}$ to $1^{-}$ in a small $x$-window, the $q$-motion emerges.
To make this explicit we will use the following convergence criteria. 

\begin{proposition}[After Theorem 11.1.4 of \cite{StroockVar}]  
\label {SVlem}
Let $a(t,z)$ and  $b(t,z)$ be continuous
from $[0, \infty) \times $ $\RR$ into $\RR$.   For each $w \in \RR$, let the solution of the martingale problem for $a$ and $b$ 
(diffusion and drift coefficients respectively) begun 
from $w$ at $t=s$ be unique.
Denote this solution by $P_{s,w}$.
Suppose next that there are $\{a_n\}$ and $\{ b_n\}$ satisfying
$$
    \sup_{n \ge 1} \sup_{t < T} \sup_{|z| < M}  ( | a_n(t,z) | + |b_n(t,z)| ) < \infty
$$
and
$$
 \lim_{n \ra \infty} \int_0^T \sup_{|z| < M}  ( | a_n(t,z) - a(t,z) | + |b_n(t,z) - b(t,z)| ) \,dt = 0
$$
for all $T>0$ and $M>0$.  Then, if $P_{s,w}^n$ is a solution of the martingale problem for 
$a_n$ and $b_n$ starting from $(s,w)$,  $P_{s,w}^n \ra P_{s,w}$.
\end{proposition}

\begin{proof}[Proof of Theorem \ref{h2sthm}] Restoring a generic value of $\mu$ we write
\begin{equation}
\label{full_p}
   d \mathfrak{p}(x)  = \stb \mathfrak{p}(x)  db(x) +  \eta^{1/2} \,
   \Bigl( 2   \mathfrak{p}(x) - \mathfrak{p}^2(x) -   (1 - \eta^{-1/3} \mu )e^{-x}  \Bigr) dx.
\end{equation}
Here $\mathfrak{p}(0) = +\infty$, while to utilize the proposition it is convenient to move the
starting point to a finite place.  

Certainly,
$$
  P_{+\infty} \Bigl(\mathfrak{p}  \mbox{ never explodes} \Bigr) 
  \ge P_{1+ \varepsilon} \Bigl( \mathfrak{p}  \mbox{ never explodes} \Bigr)
$$
for whatever $\varepsilon >0$. Also, 
$$
   P_{+\infty} \Bigl(\mathfrak{p}  \mbox{ never explodes} \Bigr)  \le P _{+\infty} \Bigl(\mathfrak{p}  
   \mbox{ never explodes},  \m_{1+\varepsilon} \le \delta \Bigr) + P_{+\infty} \Bigl( \m_{1+\varepsilon} \ge \delta  \Bigr)
$$
where $\m_{c}$ is the fist passage to the point $c$ and $\delta > 0$. By the Markov property and
monotonicity, the first term on the right is less than the $(P_{\delta, 1+\varepsilon})$-probability
of no explosion.  We wish to bound the second term from above for large $\eta$, and 
to that end note that
$P_{+\infty}(\m_a \le \m_a^{\delta}) = 1$ where $\m_a^{\delta}$ is the passage time of the 
homogeneous process $\mathfrak{p}_{\delta}$ in which the appearance of $e^{-x}$ in the
 $\mathfrak{p}$ drift is replaced by $e^{-\delta}$.  (The obvious coupling is used.) 
 Hence,
\begin{equation}
\label{cheby1}
  P_{+\infty} \Bigl( \m_{1+\varepsilon} > \delta  \Bigr) \le \frac{1}{\delta} E_{+\infty} [ \m_{1+\varepsilon}^{\delta} ]
   = \frac{1}{\delta} \int_{1+\varepsilon}^{\infty}  \int_{1+\varepsilon}^x s(dy) m(dx)
\end{equation}
for $ m(dx) $ and $s(dx)$  the  speed and scale measures of $\mathfrak{p}_{\delta}$:
$$
     m(dx) =  \frac{2}{\b} \frac{1}{x^2} e^{- \sqrt{\eta} \psi(x)}  dx,  \  \   s(dx) = e^{\sqrt{\eta} \psi(x)} dx,  \  \ 
     \psi(x) = \frac{\b}{2} \Bigl[  x - 2 \ln x - c_{\eta, \delta} \frac{1}{x} \Bigr]
$$
and $c_{\eta, \delta} = ( 1 - \mu \eta^{-1/3} ) e^{-\delta}$. Next choose
\begin{equation}
\label{scale1}
  \varepsilon= \varepsilon(\eta) = M \eta^{-1/6},  \  \  \  \delta = \delta(\eta) = \frac{1}{K} \eta^{-1/3},  
\end{equation}
where $K \ge 1$ and $M \ge \sqrt{|\mu| + 2}$.  
These last precautions imply that $\psi(x)$ is increasing for $x > 1+ \varepsilon$.  Then
an exercise in stationary phase allows the continuation of (\ref{cheby1})  as 
$$
   P_{+\infty} \Bigl( \m_{1+\varepsilon(\eta)} > \delta(\eta)  \Bigr) \le  K \eta^{1/3} \int_{1+ M \eta^{-1/6}}^{\infty}   \frac{1}{x^2}  \int_{1+ M \eta^{-1/6}}^x
 e^{- \sqrt{\eta} [ \psi(x) -  \psi(y) ]}   \, dy dx  \le  C \frac{K}{M},
 $$
for $\eta \uparrow \infty$ and a  constant $C$  depending only on $\beta$, the inner integral concentrating at the upper limit $y = x$.
In summary, for $\mathfrak{p}$ paths we have  that
\begin{equation}
\label{sandwhich}
  P_{0, 1+ \varepsilon(\eta) }  \Bigl( \m_{-\infty} = \infty \Bigr)  \le  P_{0, +\infty}  \Bigl( \m_{-\infty} = \infty  \Bigr)  
  \le   P_{\delta(\eta), 1 + \varepsilon(\eta)}  \Bigl( \m_{-\infty} = \infty   \Bigr) +  C \frac{K}{M}
\end{equation}
holds for all large $\eta$.
 
Now bring in
$$
    q_{\eta}(x) =  \eta^{1/6} \Bigl(  \mathfrak{p}( \eta^{-1/3} x)  - 1 \Bigr),
$$  
and note that, when $\mathfrak{p}$ begins at $(0, \varepsilon(\eta))$, $q_{\eta}$ begins at $(0, M)$, and when $\mathfrak{p}$ begins 
at $(\delta(\eta), \varepsilon(\eta))$, $q_{\eta}$ begins at $(K^{-1}, M)$.   Further, 
$q_{\eta}$ hits $-\infty$ if and only if $\mathfrak{p}$ does, and a substitution in (\ref{full_p}) shows that $q_{\eta}$ satisfies
the It\^o equation
$$
   d q_{\eta}(x)  =  \stb \Bigl[ 1 + \eta^{-1/6} q_{\eta}(x) \Bigr] d{\hat{b}}(x) + \Bigl[  - q_{\eta}^2(x)  +  \eta^{1/3} \Bigl(  1 - ( 1- \eta^{-1/3} \mu )  e^{-\eta^{-1/3} x}
   \Bigr)  \Bigr] dx
$$
with a new Brownian Motion $\hat{b}(x) = \eta^{1/6} b(\eta^{-1/3} x)$.   Given  unique strong solutions in both instances, 
Proposition \ref{SVlem} easily applies with
$$
   a_{\eta}(t,z) = (2/\b) [ 1 + \eta^{-1/6} z]^2 \mbox{ and }    b_{\eta}(t,z) =  [ - z^2 +    \eta^{1/3} (  1 - ( 1- \eta^{-1/3} \mu )  e^{-\eta^{-1/3} t } ) ],
$$
the $q_{\eta}$-coefficients, 
and $a(t,z) = 2/\beta$ and $b(t,z) = - z^2 + \mu + t$, the $q$-coefficients (recall (\ref{q_eq})).   
That is to say,    $\lim_{\eta \ra \infty} E_{x,c}  [ \phi (q_\eta) ] =  E_{x, c} [  \phi(q) ]$ for all bounded continuous functions of the path,
and, by approximation we also find, via (\ref{scale1}) and (\ref{sandwhich}),  that
\begin{eqnarray*}
\lefteqn{ P_{0, M} \Bigl( q \mbox{ never explodes} \Bigr)  \le  \liminf_{\eta \ra \infty} 
P_{0, \infty} \Bigl( \mathfrak{p} \mbox{ never explodes} \Bigr) } \\
                                                                       & \le &  \limsup_{\eta \ra \infty} P_{0, \infty} \Bigl( \mathfrak{p} \mbox{ never explodes} \Bigr)  \le 
                                                          P_{K^{-1}, M} \Bigl( q \mbox{ never explodes} \Bigr) +  C \frac{K}{M}.                                                                                      
\end{eqnarray*} 
Note while $q \mapsto \mathfrak{m}_{-\infty}(q)$ is not continuous,  $q \mapsto \mathfrak{m}_{-L}(q)$ is for any $L$ finite (outside a set of measure zero).
It follows that we have the distributional convergence of $ \mathfrak{m}_{-L}(q_{\eta})$ to $ \mathfrak{m}_{-L}(q)$.  The approximation
required above is then to show that: $ \lim_{L \ra \infty} \mathfrak{m}_{-L}(q) = \mathfrak{m}_{-\infty}(q)$ holds in probability, with the
same limit taking place uniformly in $\eta$ when $q_{\eta}$ replaces $q$.  That all processes involved have exit barriers at $-\infty$
makes this routine.
To finish the proof, let $M$ and then $K$ tend to infinity. The $q$-law is continuous in its initial time,
and that $\lim_{M\ra\infty} P_{c, M} = P_{c,\infty}$ 
is a byproduct  of $+\infty$ being an entrance point.
\end{proof}

\bigskip

\noindent{\bf{Acknowledgments }} We thank P. Forrester for pointing out the transition 
problem  to us, and also M. Krishnapur and T. Kurtz  for helpful input.  
The work of the second author was supported in part by 
NSF grants DMS-0505680 and DMS-0645756.

\sc \bigskip \noindent
Jos{\'e}  A. Ram{\'{\i}}rez\\
Department of Mathematics, \\
 Universidad de Costa Rica,
San Jose  2060, Costa Rica. \\
{\tt  jaramirez@cariari.ucr.ac.cr}

\sc \bigskip \noindent Brian Rider \\  Department of
Mathematics,
\\ University of Colorado at Boulder, Boulder, CO 80309. \\{\tt
brian.rider@colorado.edu}

\end{document}